\documentclass{article}
\usepackage[american]{babel}
\usepackage[utf8]{inputenc}      
\usepackage[T1]{fontenc}         
\usepackage{lmodern}  

\usepackage[en-GB]{datetime2}
\DTMlangsetup[en-GB]{ord=omit}
\AtBeginDocument{\DTMsetstyle{en-GB}}

\usepackage{silence}
\usepackage{amsmath,amssymb,amsfonts,amsthm,thmtools,mathrsfs,xcolor, csquotes, microtype,float,mathtools}
\usepackage[margin=35mm]{geometry}

\usepackage[backend=biber,style=alphabetic,giveninits=true,isbn=false,url=false]{biblatex}

\DeclareDelimFormat{postnotedelim}{:\addnbspace}
\ExplSyntaxOn
\cs_new_protected:Npn \ReplaceSpacesWithTies #1
  {
    \tl_set:Nn \l_tmpa_tl { #1 }
    \tl_replace_all:Nnn \l_tmpa_tl { ~ } { \nobreakspace }
    \tl_use:N \l_tmpa_tl
  }
\ExplSyntaxOff

\DeclareFieldFormat{postnote}{\ReplaceSpacesWithTies{#1}}
\DeclareFieldFormat{multipostnote}{\ReplaceSpacesWithTies{#1}}

\bibliography{references.bib}

\usepackage{tikz}
\usetikzlibrary{cd,bbox,calc,intersections}
\usetikzlibrary{decorations.markings}
\tikzset{->-/.style={postaction={decorate}, decoration={
  markings,
  mark=at position #1 with {\arrow{>}}}}}
\tikzset{-<-/.style={postaction={decorate}, decoration={
  markings,
  mark=at position #1 with {\arrow{<}}}}}
\usepackage[justification=centerlast,margin=20pt]{caption}

\definecolor{color1}{rgb}{0.4, 0.8, 0.1}
\definecolor{color2}{rgb}{0.3, 0.5, 0.8}
\definecolor{color3}{RGB}{255, 166, 0}
\definecolor{color4}{rgb}{0.8, 0, 0}
\definecolor{color5}{rgb}{0.3, 0.9, 0.8}

\tikzset{dot/.pic={\filldraw circle[radius=0.5mm];}}

\definecolor{forlinks}{rgb}{0.55, 0.0, 0.0}
\usepackage[pageanchor,colorlinks=true,allcolors=forlinks]{hyperref}
\usepackage[capitalize,nameinlink,noabbrev]{cleveref}

\usepackage{enumitem}
\setlist{nosep,font=\normalfont}
\setlist[enumerate,1]{label=\normalfont{(\arabic*)}} 
\setlist[itemize,1]{label=\normalfont\bfseries\textendash}   

\newtheorem{thm}{Theorem}[section]

\newtheorem{cor}[thm]{Corollary}
\newtheorem{lem}[thm]{Lemma}
\newtheorem{prop}[thm]{Proposition}

\theoremstyle{definition}

\newtheorem{manualinner}{} 
\newenvironment{manual}[1]{%
  \renewcommand\themanualinner{#1}%
  \manualinner
}{\endmanualinner}

\newcommand{\ZZ}{\mathbb{Z}}
\newcommand{\RR}{\mathbb{R}}
\newcommand{\QQ}{\mathbb{Q}}
\newcommand{\OD}{\mathring{\mathbb{D}}}
\newcommand{\DD}{\mathbb{D}}

\newcommand{\h}{\operatorname{H}}

\newcommand{\rev}{^\mathrm{rev}}

\renewcommand{\emph}[1]{\textit{\textbf{#1}}}

\title{\textbf{Cobordism-equivalence for\\codimension-one submanifolds}}
\author{Stefan Friedl$^1$ \and Tobias Hirsch$^1$ \and Clayton McDonald$^2$ \and José Pedro Quintanilha$^3$\and Daniel Zach$^1$}
\date{$^1$Universität Regensburg \qquad $^2$Renyi Institute \qquad $^3$Universität Heidelberg\\\bigskip\today
\footnotemark[0]
\footnotetext[0]{2020 \textit{Mathematics Subject Classification}: 57K10; 57R65, 57R90}}

\begin{document}

\maketitle

\begin{abstract}
    We show that two hypersurfaces in a manifold are related by a sequence of embedded cobordisms if and only if they represent the same homology class. By applying handle decompositions we turn these cobordisms into a sequence of embedded surgeries. Specializing to Seifert surfaces we obtain a conceptual proof that two Seifert surfaces of a fixed link are related by tube attachments and tube removals.
\end{abstract}

\section{Introduction}






It is a well-known result in knot theory that two Seifert surfaces of a given link are related by a sequence of tube attachments and tube removals. This is a key input for showing that two Seifert matrices of the same link are S-equivalent. There are many different types of proofs in the literature,
see e.g.\
\cite[Theorem 11]{GLi1978},  
\cite[p.197]{Kau1987a},  
\cite[Theorem 5.2.3]{Kaw1996},  
\cite[Theorem 8.2]{Lic1997a},   
\cite[Proposition 4.1.5, Theorem 5.1.11]{Kai1997}, 
\cite{BFK},
\cite[Theorem 5.5.1]{Crom2004},  
\cite[Theorem 5.4.1]{Muras2008}, and   
\cite[Proposition 4.11]{Juh2023}. 
Many of these proofs are ad-hoc, relying heavily on tools specific to knot theory such as the Reidemeister moves theorem and Seifert's algorithm. Some might also be better described as sketches with weaknesses especially in the case of links. We offer an alternative, more conceptual, approach by proving a general statement about manifolds in all dimensions which, with a bit of extra care, gives the knot-theoretic claim when restricting to link exteriors. As a consequence, our proof also applies to Seifert surfaces in rational homology spheres.

We proceed in two steps. The first is our main result, which asserts that homology is the only obstruction to two proper codimension-1 submanifolds being related by a sequence of embedded cobordisms:

\begin{manual}{\cref{thm:main}}
   Let $A$, $B$ be oriented proper codimension-1 smooth submanifolds of an oriented $n$-dimensional smooth manifold $M$. Then $A$ and $B$ represent the same class in $\h_{n-1}(M,\partial M)$ if and only if there exists a sequence \[A=S_0,S_1,\dots,S_k=B\] of  oriented proper codimension-1 smooth submanifolds of $M$ such that there is an embedded cobordism in $M$ from $S_{i-1}$ to $S_i$.
\end{manual}
It is not hard to see that the homology class obstructs such a sequence, see \cref{lem:easymain}. For the reverse direction there are two potential obstacles. Firstly, if $A$ and $B$ are not disjoint, it is not sensible to speak of an embedded cobordism between them. We overcome this with previous work by Herrmann and the fourth author \cite[Theorem 1.1]{HQ25}, which we recall with some modifications in \cref{thm:HQ}. Afterwards, we may assume $A\cap B=\emptyset$, but have lost count of the number of components of $A$ and $B$. The second hurdle therefore is that, unless $A$ and $B$ are connected, no union of components of $M\setminus(A\cup B)$ needs to define the desired cobordism from~$A$ to~$B$.
\begin{figure}[H]
    \centering
    \begin{tikzpicture}
        \draw[thick] (0,0) circle(1.3);
        \draw[color1,ultra thick] (0,0) +(90:1.3) node[below]{$A$} pic{dot} +(60:1.3) pic{dot} +(120:1.3) pic{dot} ;
        \draw[color2,ultra thick,yscale=-1] (0,0) +(90:1.3) node[above]{$B$} pic{dot} +(60:1.3) pic{dot} +(120:1.3) pic{dot} ;
    \end{tikzpicture}\qquad\qquad
    \begin{tikzpicture}
    \fill[black!05,rounded corners] (-1.5,-1.4) rectangle (1.5,1.4);
    \foreach \y in {-1,0,1} {
        \draw[ultra thick, color1] (0,\y) -- +(1.5,0); 
        \draw[ultra thick, color2] (0,\y) -- +(-1.5,0); 
        \draw (0,\y) pic[ultra thick]{dot};
        \draw[color1,->] (1.3,\y) -- +(0,-0.25);
        \draw[color2,->] (-1.3,\y) -- +(0,0.25);
    }
    \path (1.3,0) node[above,color1]{$A$} (-1.3,0) node[below,color2]{$B$};
    \end{tikzpicture}
    \caption{Left: Two collections $A$, $B$ of three points on a circle, all with the same orientation. Their complement has 6 components, no combination of which forms an embedded cobordism from $A$ to $B$.\\    
    Right: The same idea realized knot-theoretically. We see a cross-section of three parallel copies of a knot $K$ protruding from the drawing layer with two surfaces $A$, $B$. Each surface is formed by three parallel copies of a Seifert surface of $K$, all six of these are pairwise disjoint. Their orientations are indicated by normal vectors.\\    
    The complement of $A\cup B$ has four components, no combination of which forms an embedded cobordism from $A$ to $B$.} 
    \label{fig:exampleknot}
\end{figure}
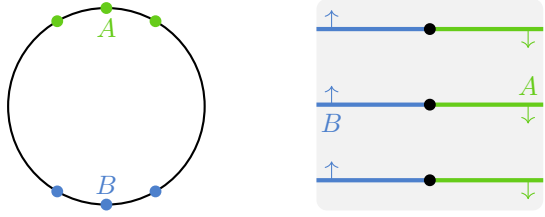
This is overcome by a graph-theoretic argument which shows that there exists a sequence of cobordisms, each of which is a component of $M\setminus(A\cup B)$, going from $A$ to $B$. 

As a second step we apply the theory of handle decompositions for cobordisms to the cobordisms from the main theorem. We recall it in \cref{sec:handles}. This yields a sequence of embedded surgeries from $A$ to $B$. Specializing to the \mbox{3-dimensional} case we produce a sequence of tube attachments and tube removals, thus obtaining the motivating knot-theoretic result as a corollary. Here a \emph{Seifert surface} for an oriented link $L$ in a rational homology sphere $X$ is a connected compact oriented 2-dimensional smooth submanifold $\Sigma\subseteq X$ with $\partial\Sigma=L$. 
\begin{manual}{\cref{cor:seifert_tube}}
    Let $A$, $B$ be Seifert surfaces for a link $L$ in a rational homology sphere $X$. Then $B$~arises from~$A$ by a sequence of smooth isotopies, tube at\-tach\-ments and tube removals with the tube attachments occurring before the tube removals. In particular, all intermediate surfaces are also Seifert surfaces.
\end{manual}
\paragraph{Conventions.} All manifolds are compact oriented and smooth, possibly with boundary. If a manifold has corners, it will be explicitly called a manifold with corner.

\paragraph{Acknowledgments.} Initial discussion for this article took place during the workshop on \emph{Surfaces in 4-manifolds}, hosted at Universität Regensburg in December 2025, supported by the SFB 1085 Higher Invariants (Universität Regensburg, funded by the DFG, ID 224262486) and organized by Marc Kegel, Mark Powell, Arunima Ray and the first author.
\section{Preliminaries}
\subsection{Bipartitions and pushing sinks}\label{sec:graphs}
Before we come to topology, we establish two lemmas from graph theory. We will not pick a precise model for the definition of \enquote{graph}, \enquote{cycle} in a graph, \enquote{orientation} of a graph, etc. We are sure the reader can make the required notions precise, for example by following \cite[Section 2.1]{serre}. We allow graphs to have parallel edges and loops, i.e.\ edges between a vertex and itself. A graph shall always mean an unoriented object. If we need an orientation, this will be specified.

A finite graph $\Gamma$ is \emph{bipartite} if its set of vertices is a disjoint union of two subsets $L,R$ such that every edge lies between a vertex in $L$ and one in $R$. Our first lemma is a well-known criterion for a graph to be bipartite.
\begin{lem}[Criterion for bipartitions]\label{lem:crit_bip}
    A finite graph $\Gamma$ is bipartite if and only if all its cycles have even length. 
\end{lem}
\begin{proof}
    If $\Gamma$ is bipartite, the cycles have even length as their vertices must alternate between $L$ and $R$.

    If all cycles have even length, we can define a bipartition by arbitrarily assigning one vertex in every component to $L$ and then successively propagating to neighboring vertices.
\end{proof}

Now consider a finite graph $\Gamma$ with an orientation $O$. A \emph{sink} of $O$ is a vertex with only incoming edges, a \emph{source} is a vertex with only outgoing edges. An orientation $O'$ of $\Gamma$ \emph{arises from $O$ by pushing up a sink} if $O$ and $O'$ differ precisely on the edges surrounding a sink of~$O$. In other words, a sink of $O$ has become a source in $O'$. The next lemma by Pretzel \cite{Pretzel} describes which orientations can be reversed by pushing up sinks.

\begin{lem}[Reversing by pushing up sinks]\label{lem:sinkpushup}
    Let $\Gamma$ be a finite graph with an orientation $O$. The following are equivalent:
    \begin{enumerate}
        \item For every cycle $\gamma$ in $\Gamma$ the number of edges $\gamma$ traverses in the direction of~$O$ equals the number of edges $\gamma$~traverses against it.
        \item There exists a sequence of orientations $O=O_0,O_1,\dots,O_k=O\rev$ such that $O_i$~arises from~$O_{i-1}$ by pushing up a sink.
    \end{enumerate}
\end{lem}
\begin{proof} 
    It suffices to prove the lemma when $\Gamma$ is connected.
    
    Suppose (1) holds. Then $O$ allows a \emph{height} map $h\colon V(\Gamma)\to\ZZ$, i.e.\ a map on the vertices of~$\Gamma$ such that $h(v')=h(v)-1$ if there is an oriented edge from $v$ to~$v'$. 
    We can define this map by arbitrarily assigning the value of one vertex and then successively propagating to neighboring vertices. Condition (1) precisely ensures that this is well-defined. We may assume after an appropriate shift that all heights are at most 0 and that there exists a vertex of height 0. Observe the following:
    \begin{itemize}
        \item A height function for the orientation obtained by pushing up a sink results from adding 2 to the height of the sink.   
        \item A height function determines the orientation it arises from. In particular, $h\rev\coloneqq -h$ arises from $O\rev$.
    \end{itemize}
    With this in mind, we describe a procedure for obtaining the sequence in (2):

    Start in step $i\coloneqq 0$ with $h_0\coloneq h$ and $O_0\coloneq O$. 
    Consider the set \[S_i\coloneq \{v\in V\mid h_i(v)\neq h\rev(v)\}.\] If this set is empty, stop. Otherwise take a vertex $s$ of minimal height in $S_i$ and push it up, thereby obtaining a new height $h_{i+1}$ and a new orientation $O_{i+1}$. 
    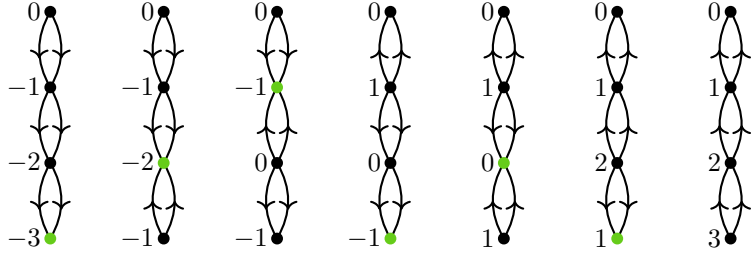
\begin{figure}[H]
        \centering
                \begin{tikzpicture}
            \draw[thick,-<-=0.5] (0,-3) to[bend left=30] (0,-2);
            \draw[thick,-<-=0.5] (0,-3) to[bend right=30] (0,-2);
            \draw[thick,-<-=0.5] (0,-2) to[bend left=30] (0,-1);
            \draw[thick,-<-=0.5] (0,-2) to[bend right=30] (0,-1);
            \draw[thick,-<-=0.5] (0,-1) to[bend left=30] (0,0);
            \draw[thick,-<-=0.5] (0,-1) to[bend right=30] (0,0);
            
            \draw (0,-3) pic[ultra thick,color1]{dot} node[left]{$-3$};
            \draw (0,-2) pic[ultra thick]{dot} node[left]{$-2$};
            \draw (0,-1) pic[ultra thick]{dot} node[left]{$-1$};
            \draw (0,0) pic[ultra thick]{dot} node[left]{$0$};
        
            \tikzset{shift={(1.5,0)}}
            \draw[thick,->-=0.5] (0,-3) to[bend left=30] (0,-2);
            \draw[thick,->-=0.5] (0,-3) to[bend right=30] (0,-2);
            \draw[thick,-<-=0.5] (0,-2) to[bend left=30] (0,-1);
            \draw[thick,-<-=0.5] (0,-2) to[bend right=30] (0,-1);
            \draw[thick,-<-=0.5] (0,-1) to[bend left=30] (0,0);
            \draw[thick,-<-=0.5] (0,-1) to[bend right=30] (0,0);
            
            \draw (0,-3) pic[ultra thick]{dot} node[left]{$-1$};
            \draw (0,-2) pic[ultra thick,color1]{dot} node[left]{$-2$};
            \draw (0,-1) pic[ultra thick]{dot} node[left]{$-1$};
            \draw (0,0) pic[ultra thick]{dot} node[left]{$0$};
        
            \tikzset{shift={(1.5,0)}}
            \draw[thick,-<-=0.5] (0,-3) to[bend left=30] (0,-2);
            \draw[thick,-<-=0.5] (0,-3) to[bend right=30] (0,-2);
            \draw[thick,->-=0.5] (0,-2) to[bend left=30] (0,-1);
            \draw[thick,->-=0.5] (0,-2) to[bend right=30] (0,-1);
            \draw[thick,-<-=0.5] (0,-1) to[bend left=30] (0,0);
            \draw[thick,-<-=0.5] (0,-1) to[bend right=30] (0,0);
            
            \draw (0,-3) pic[ultra thick]{dot} node[left]{$-1$};
            \draw (0,-2) pic[ultra thick]{dot} node[left]{$0$};
            \draw (0,-1) pic[ultra thick,color1]{dot} node[left]{$-1$};
            \draw (0,0) pic[ultra thick]{dot} node[left]{$0$};
        
            \tikzset{shift={(1.5,0)}}
            \draw[thick,-<-=0.5] (0,-3) to[bend left=30] (0,-2);
            \draw[thick,-<-=0.5] (0,-3) to[bend right=30] (0,-2);
            \draw[thick,-<-=0.5] (0,-2) to[bend left=30] (0,-1);
            \draw[thick,-<-=0.5] (0,-2) to[bend right=30] (0,-1);
            \draw[thick,->-=0.5] (0,-1) to[bend left=30] (0,0);
            \draw[thick,->-=0.5] (0,-1) to[bend right=30] (0,0);
            
            \draw (0,-3) pic[ultra thick,color1]{dot} node[left]{$-1$};
            \draw (0,-2) pic[ultra thick]{dot} node[left]{$0$};
            \draw (0,-1) pic[ultra thick]{dot} node[left]{$1$};
            \draw (0,0) pic[ultra thick]{dot} node[left]{$0$};
        
            \tikzset{shift={(1.5,0)}}
            \draw[thick,->-=0.5] (0,-3) to[bend left=30] (0,-2);
            \draw[thick,->-=0.5] (0,-3) to[bend right=30] (0,-2);
            \draw[thick,-<-=0.5] (0,-2) to[bend left=30] (0,-1);
            \draw[thick,-<-=0.5] (0,-2) to[bend right=30] (0,-1);
            \draw[thick,->-=0.5] (0,-1) to[bend left=30] (0,0);
            \draw[thick,->-=0.5] (0,-1) to[bend right=30] (0,0);
            
            \draw (0,-3) pic[ultra thick]{dot} node[left]{$1$};
            \draw (0,-2) pic[ultra thick,color1]{dot} node[left]{$0$};
            \draw (0,-1) pic[ultra thick]{dot} node[left]{$1$};
            \draw (0,0) pic[ultra thick]{dot} node[left]{$0$};
        
            \tikzset{shift={(1.5,0)}}
            \draw[thick,-<-=0.5] (0,-3) to[bend left=30] (0,-2);
            \draw[thick,-<-=0.5] (0,-3) to[bend right=30] (0,-2);
            \draw[thick,->-=0.5] (0,-2) to[bend left=30] (0,-1);
            \draw[thick,->-=0.5] (0,-2) to[bend right=30] (0,-1);
            \draw[thick,->-=0.5] (0,-1) to[bend left=30] (0,0);
            \draw[thick,->-=0.5] (0,-1) to[bend right=30] (0,0);
            
            \draw (0,-3) pic[ultra thick,color1]{dot} node[left]{$1$};
            \draw (0,-2) pic[ultra thick]{dot} node[left]{$2$};
            \draw (0,-1) pic[ultra thick]{dot} node[left]{$1$};
            \draw (0,0) pic[ultra thick]{dot} node[left]{$0$};
        
            \tikzset{shift={(1.5,0)}}
            \draw[thick,->-=0.5] (0,-3) to[bend left=30] (0,-2);
            \draw[thick,->-=0.5] (0,-3) to[bend right=30] (0,-2);
            \draw[thick,->-=0.5] (0,-2) to[bend left=30] (0,-1);
            \draw[thick,->-=0.5] (0,-2) to[bend right=30] (0,-1);
            \draw[thick,->-=0.5] (0,-1) to[bend left=30] (0,0);
            \draw[thick,->-=0.5] (0,-1) to[bend right=30] (0,0);
            
            \draw (0,-3) pic[ultra thick]{dot} node[left]{$3$};
            \draw (0,-2) pic[ultra thick]{dot} node[left]{$2$};
            \draw (0,-1) pic[ultra thick]{dot} node[left]{$1$};
            \draw (0,0) pic[ultra thick]{dot} node[left]{$0$};        
        \end{tikzpicture}
        \caption{The sequence of orientations when applying this procedure to a doubled line.}
        \label{fig:examplegraph}
    \end{figure}
    By induction $h_{i}(v)\leq h\rev(v)$ for all $i\geq0$ and vertices $v$. It remains to observe the following:
    \begin{itemize}
        \item For $i\geq0$, a vertex $s\in S_i$ of minimal height is a sink in $O_i$:
        
        Suppose it is not. Then there exists an edge $e$ from $s$ to some vertex $v$. By minimality, $v\notin S_i$, hence, $h_i(v)=h\rev(v)$ and $h_i(s)\in\{h\rev(s),h\rev(s)+2\}$ depending on the orientation of $e$ in the original orientation $O$. The first one is impossible as $s\in S$, the second as $h_i(s)\leq h\rev (s)$. 
        \item The procedure stops since in every step $i\geq0$ the non-negative integer
        \[
        \sum_{v\in V} h\rev(v)-h_i(v)
        \]
        decreases by 2. If it stops in step $i$, we have $S_i=\emptyset$ which implies  $O_i=O\rev$. 
    \end{itemize}
    
    For the reverse implication note that for every cycle $\gamma$ the difference between the number of edges $\gamma$~traverses in the direction of $O$ and the number of edges $\gamma$ traverses against it remains unchanged when pushing up a sink. Hence, if such a sequence exists, this quantity must be the same in $O$ and $O\rev$, implying that it is zero.
\end{proof}
\subsection{Corners and cobordisms}\label{sec:handles}
We will need cobordisms between manifolds with boundary, which are manifolds with corner. We recall this terminology following \cite[pp.30,130]{Wall}:

A \emph{chart} for a point $P$ in a topological space $W$ is a homeomorphism $\Phi\colon U\to V$ with $\Phi(P)=0$ where $U$ is an open neighborhood of $P$ and $V$ is an open neighborhood of $\Phi(P)=0$ in $\RR^n$, $\RR_{\geq0}\times\RR^{n-1}$ or $\RR_{\geq0}^2\times\RR^{n-2}$. An \emph{$n$-dimensional smooth manifold with corner} is a second-countable Hausdorff space $W$ together with a collection of smoothly compatible charts containing a chart for every $P\in W$. The \emph{boundary} $\partial M$ of $M$ consists of all points with a chart in $\RR_{\geq0}\times\RR^{n-1}$ or $\RR_{\geq0}^2\times\RR^{n-2}$, the \emph{corner} $\angle W$ of all points with a chart in $\RR_{\geq0}^2\times\RR^{n-2}$.

Let $W$ be an oriented $n$-dimensional smooth manifold with corner. A closed subset $A\subseteq\partial W$ is a \emph{smooth boundary of $W$} if for every $P\in A$ there exists a chart $\Phi\colon U\to V$ with $\Phi(U\cap A)=V\cap(\{0\}\times\RR^{n-1})$. Then $A$ is naturally an oriented smooth manifold with boundary $\partial A=A\cap\angle W$. 

Two $(n-1)$-dimensional oriented manifolds with boundary $A$, $B$ are \emph{abstractly cobordant} if there exist an $n$-dimensional manifold with corner $W$, smooth boundaries $\tilde A,\tilde B,V\subseteq\partial W$ and diffeomorphisms $\alpha\colon A\to\tilde A$, $\beta\colon B\to\tilde B$ such that
\begin{itemize}
    \item $\tilde A$ and $\tilde B$ are disjoint, 
    \item $\partial V=\partial(\tilde A\cup\tilde B)=\angle W$, and
    \item $\partial W=\tilde B\cup \tilde A\rev\cup V$ as an oriented manifold.
\end{itemize}
In this case, $W$ is an \emph{abstract cobordism} from $A$ to $B$ with \emph{vertical part} $V$.

If the vertical boundary is a product, i.e.\ if there exists a diffeomorphism $\phi\colon \partial A\times[0,1]\to V$ restricting to $\alpha$ on $\partial A\times\{0\}$, one may develop handle decompositions as usual. Indeed, a \emph{$k$-handle attachment} is then given by 
\[
W\cup_f h^k\coloneqq W\cup_f\!\left(\DD^k\times\DD^{n-k}\right)
\]
where $f\colon\partial\DD^k\times\DD^{n-k}\to \tilde B\setminus\partial\tilde B$ is a smooth embedding. We view $W\cup_f h^k$ again as a cobordism starting at $A$ with vertical part $V\cong\partial A\times[0,1]$, on which we can iterate this process. This leads to \emph{handle decomposition}, i.e.\ a diffeomorphism 
\[
\Phi\colon(A\times[0,1])\cup h^{i_1}\cup\dots\cup h^{i_k}\to W
\]
that restricts to $\alpha$ on $A\times\{0\}$ and $\phi$ on $\partial A\times[0,1]$. One may use the usual Morse-theoretic approach to show the existence of handle decompositions for cobordisms, see \cite[Lemma 5.1.1, Corollary 5.1.7]{Wall}:
\begin{thm}[Existence of handle decompositions]\label{thm:handle_decomp}
Let $W$ be a compact $n$\nobreakdash-dimensional cobordism from $A$ to $B$ with vertical part $V$ such that there exists a diffeomorphism $\phi\colon\partial A\times[0,1]\to V$ which is the identity on $\partial A\times\{0\}$. Then there exists a handle decomposition 
\[
(A\times[0,1])\cup h^{i_1}\cup\dots\cup h^{i_k}\to W
\]
extending $\phi$.
\end{thm}
We will also want to consider embedded cobordisms between submanifolds of a manifold with boundary $M$. This leads to the notion of submanifolds with corner (compare \cite[p.31]{Wall}). A subset $W\subseteq M$ is a \emph{$k$-dimensional submanifold with corner} if each point $P\in W$ has a neighborhood $U\subseteq M$ with a smooth embedding $\phi\colon U\to\RR^n$ such that $\varphi(P)=0$ and one of the following holds:
\begin{enumerate}
    \item $\varphi(U)\subseteq\RR^n$ is open and $\varphi(U\cap W)=\varphi(U)\cap(\RR^k\times\{0\})$,
    \item $\varphi(U)\subseteq\RR_{\geq0}\times\RR^{n-1}$ is open and $\varphi(U\cap W)=\varphi(U)\cap(\RR^k\times\{0\})$,
    \item $\varphi(U)\subseteq\RR^n$ is open and $\varphi(U\cap W)=\varphi(U)\cap(\RR_{\geq0}\times\RR^{k-1}\times\{0\})$,
    \item $\varphi(U)\subseteq\RR_{\geq0}\times\RR^{n-1}$ is open and $\varphi(U\cap W)=\varphi(U)\cap(\RR_{\geq0}\times\RR^{k-1}\times\{0\})$.
\end{enumerate}
Then $W$ is naturally a $k$-dimensional smooth manifold with corner. We can decompose $\partial W$, which consists of the points having a chart of type (2)--(4), into
\begin{itemize}
    \item $\partial_v W$, given by the closure of the points having a chart of type (2) in $W$ or equivalently $W\cap\partial M$, and
    \item $\partial_\circ W$, given by the closure of the points having a chart of type (3) in $W$ or equivalently the closure of $\partial W\cap\mathring M$ in $W$. 
\end{itemize}
These sets intersect in the corner $\angle W$ consisting of the points with a chart of type (4). Also note that $\partial_v W$ and $\partial_\circ W$ are smooth boundaries of $W$.
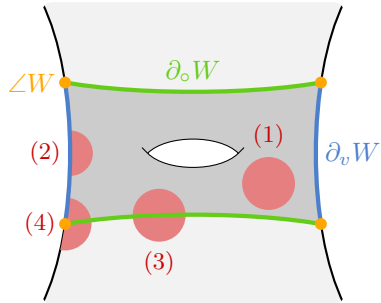
\begin{figure}[H]
    \centering
    \begin{tikzpicture}
        \path[use as bounding box] (-2,-2) rectangle (2,2);
        
        \begin{scope}
        \clip[rounded corners] (-2,-2) rectangle (2,2);
        \fill[black!05] (-2,2) .. controls +(0.5,-1) and +(0.5,1) .. (-2,-2) --
        (2,-2) .. controls +(-0.5,1) and +(-0.5,-1) .. (2,2) -- cycle;
        \draw[thick,name path=lr] (-2,2) .. controls +(0.5,-1) and +(0.5,1) .. (-2,-2)
        (2,-2) .. controls +(-0.5,1) and +(-0.5,-1) .. (2,2);  
        \path[name path=tb,rotate=90,xscale=0.5] (-2,2) .. controls +(0.5,-1) and +(0.5,1) .. (-2,-2)
        (2,-2) .. controls +(-0.5,1) and +(-0.5,-1) .. (2,2);  
        \path[name intersections={of=tb and lr, name=c}];
        
        \path[name path=v] (0,-2) -- (0,2);
        \path[name path=h] (-2,0) -- (2,0);               
        \path[name intersections={of=tb and v, name=v}];
        \path[name intersections={of=lr and h, name=h}];

        \begin{scope}
            \clip (-2,2) .. controls +(0.5,-1) and +(0.5,1) .. (-2,-2) --
            (2,-2) .. controls +(-0.5,1) and +(-0.5,-1) .. (2,2) -- cycle; 
            \begin{scope}
                \clip[rotate=90,xscale=0.5] (-2,2) .. controls +(0.5,-1) and +(0.5,1) .. (-2,-2) node[pos=0.4] (A) {} -- (2,-2) .. controls +(-0.5,1) and +(-0.5,-1) .. (2,2) -- cycle ; 
                \fill[black!20] (-2,-2) rectangle (2,2);   
            \end{scope} 
            \fill[color4!50] (A) circle(0.35) (h-1) circle(0.3)  (c-1) circle(0.35) (1,-0.4) circle(0.35); 
        \end{scope}  
        \begin{scope}
            \clip (-2,2) .. controls +(0.5,-1) and +(0.5,1) .. (-2,-2) --
            (2,-2) .. controls +(-0.5,1) and +(-0.5,-1) .. (2,2) -- cycle;            
            \draw[ultra thick,rotate=90,xscale=0.5,color1] (-2,2) .. controls +(0.5,-1) and +(0.5,1) .. (-2,-2)
            (2,-2) .. controls +(-0.5,1) and +(-0.5,-1) .. (2,2);               
        \end{scope}
        \begin{scope}
            \clip[rotate=90,xscale=0.5] (-2,2) .. controls +(0.5,-1) and +(0.5,1) .. (-2,-2) --
            (2,-2) .. controls +(-0.5,1) and +(-0.5,-1) .. (2,2) -- cycle;           
            \draw[ultra thick,color2] (-2,2) .. controls +(0.5,-1) and +(0.5,1) .. (-2,-2)
            (2,-2) .. controls +(-0.5,1) and +(-0.5,-1) .. (2,2);               
        \end{scope}           
        
        \filldraw[fill=white,draw=black] (-0.6,0) .. controls +(0.25,0.25) and +(-0.25,0.25) .. (0.6,0) .. controls +(-0.25,-0.25) and +(0.25,-0.25) .. (-0.6,0);
        \draw (0.6,0) -- +(0.075,0.075) (-0.6,0) -- +(-0.075,0.075);

        \foreach \i in {1,...,4}{
            \draw[color3] (c-\i) pic[ultra thick]{dot};
        } 
       
        \end{scope}                

        \path (c-3) node[left,color3]{$\angle W$} (v-2) node[above,color1]{$\partial_\circ W$} (h-2) node[right,color2]{$\partial_vW$} (c-1) node[left,color4]{\small (4)} (A) ++(0,-0.35) node[below,color4]{\small (3)} (h-1) node[left,color4]{\small (2)} (1,-0.05) node[above,color4]{\small (1)};
    \end{tikzpicture}
\end{figure}

\section{Cobordisms of hypersurfaces}
We now turn to cobordisms of hypersurfaces. Let $M$ be a (compact oriented smooth) manifold of dimension $n$ with (possibly empty) boundary. 
A codimension-1 submanifold $S$ of $M$ is a \emph{hypersurface} if it is oriented and \emph{proper}, i.e.\ every point of $S$ has a chart of type (1) or (2). We do not require $M$ or $S$ to be connected. 
In this situation, we will denote by $[M]\in \h_n(M, \partial M)$ the fundamental class of $M$ (with integer coefficients) and, with slight abuse, write $[S] \in \h_{n-1}(M, \partial M)$ to denote the image of the fundamental class of $S$ under the inclusion-induced map $\h_{n-1}(S, \partial S) \to \h_{n-1}(M, \partial M)$. We will write~$S\rev$ to denote the hypersurface $S$ with reversed orientation.

Given two disjoint hypersurfaces $A, B$ in $M$, an \emph{(embedded) cobordism}~$W$ from $A$ to $B$ is a compact codimension-$0$ submanifold with corner of~$M$, such that
\[\partial_\circ W = B \cup A\rev.\]
with the orientation on $W$ induced by $M$. It is sometimes convenient to allow hypersurfaces with common components. Hence, if $A$, $B$ are hypersurfaces in $M$ such that $I\coloneq  A\cap B$ is a union of components of $A$ (and~$B$) with matching orientation, a cobordism from $A\setminus{}I$ to~$B\setminus{}I$ that does not intersect $I$ shall also be a \emph{cobordism} from $A$ to $B$. 

We emphasize the distinction between this definition, where both the cobordism and the hypersurfaces are understood as embedded objects in $M$, and the notion of abstract cobordism of manifolds discussed in the last section. As we will only consider embedded cobordisms from now on, we will drop the word \enquote{embedded}.

A cobordism must always be the closure of a union of components of $M\setminus(A\cup B)$. However, a component of $M\setminus(A\cup B)$ only defines a cobordism if its orientation aligns in the following sense. A component $v$ of $M\setminus (A\cup B)$ \emph{touches the positive side of $c$}, where $c$ is a union of components of $A\cup B$, if 
\[
v\cap\beta(c\times[-1,1])\subseteq\beta(c\times[0,1])
\]
for some \emph{bicollar} $\beta\colon c\times[-1,1]\to M$, i.e.\ an orientation-preserving smooth embedding that restricts to the identity on $c\times\{0\}$. Analogously, we define \emph{touching the negative side}. 
\begin{lem}[Cobordisms from components]\label{lem:comp_to_cob}
    Let $A$, $B$ be disjoint hypersurfaces in an $n$-dimensional manifold $M$ and $v$ a connected component of $M\setminus(A\cup B)$. The closure of $v$ in $M$ defines a cobordism from $A$ to $B$ if and only if $v$ touches the positive side of $A$ and the negative side of $B$.
\end{lem}
\begin{proof}
    Let $W$ be the closure of $v$ in $M$. Then $W\setminus v\subseteq A\cup B$. If $W$ is a cobordism from $A$ to $B$, we have $\partial_\circ W=B\cup A\rev$ which implies that $v$ touches the positive side of $A$ and the negative side of $B$.
    
    Now suppose  $v$ touches the positive side of $A$ and the negative side of $B$. We first need to show that $W\subseteq M$ is a submanifold with corner. The points in $v\subseteq W$ have charts of type (1) or~(2) as $v\subseteq M$ is open. The bicollars yield charts of type (3) or~(4) for the points in $W\setminus v\subseteq A\cup B$. Hence, $W\subseteq M$ is a submanifold with corner and $\partial_\circ W=A\cup B$. As $v$ lies on the positive side of $A$ and the negative side of $B$, $W$ defines a cobordism from $A$ to $B$ with the orientation induced by $M$.
\end{proof}

Cobordism generates an equivalence relation on the set of all hypersurfaces in $M$. In other words, $A$ and $B$ are \emph{cobordism-equivalent} if there exist hypersurfaces
\[A = S_0, S_1, \ldots, S_k = B\]
such that each two consecutive $S_{i-1}, S_i$ are cobordant. 

With the terminology established, we can immediately prove the easy direction of our main theorem:
\begin{lem}[Homology obstruction]\label{lem:easymain}
    Let $A,B$ be cobordism-equivalent hypersurfaces in an $n$-dimensional manifold $M$. Then $[A]=[B]\in\h_{n-1}(M,\partial M)$
\end{lem}
\begin{proof}
    We may assume that $A,B$ are cobordant, so let $W$ be a cobordism from~$A$ to~$B$. To see that $W$ exhibits the classes $[A], [B]$ as homologous, consider the long exact sequence of the triple $(W, \partial W, \partial W \cap \partial M)$ which contains the top row of the following diagram.
    \[\begin{tikzcd}
        \h_n(W, \partial W) \ar[r,"\partial"] & \h_{n-1}(\partial W, \partial W \cap \partial M) \ar[r] & \h_{n-1} (W,\partial W \cap \partial M)\ar[d]\\
        & \h_{n-1}(B \cup A\rev, \partial B \cup \partial A\rev)\ar[u, "\cong", "\text{excision}"']\ar[r]& \h_{n-1}(M, \partial M)
    \end{tikzcd}\]
    The unlabeled arrows are inclusion-induced maps, so the diagram commutes.

    Let $[W]\in \h_{n}(W, \partial W)$ be the fundamental class. The connecting homomorphism $\partial$ sends~$[W]$ to the class $\omega\in \h_{n-1}(\partial W, \partial W \cap \partial M)$ represented by $\partial W$. In particular, the top-right map sends $\omega$ to $0$. 
    
    On the other hand, $\omega$ corresponds under the excision isomorphism to the fundamental class $[B]-[A] \in \h_{n-1}(B \cup A\rev,\partial B \cup \partial A\rev)$, and so in $\h_{n-1}(M, \partial M)$ we have $[B]-[A] =0$.
\end{proof}

We will later want to consider handle decompositions for our cobordisms. To apply \cref{thm:handle_decomp} we will need to assume that they are \emph{vertical in the boundary} in the sense that $\partial_vW$ is a union of product cobordisms from~$\partial A$ to~$\partial B$ with closed components of~$\partial M$. Cobordism with vertical boundary induces another equivalence relation on the set of all hypersurfaces in~$M$. In other words, two hypersurfaces $A$, $B$ in $M$ are \emph{cobordism-equivalent with vertical boundary} if there exist hypersurfaces 
\[
A=S_0,S_1,\dots,S_k=B
\]
such that each two consecutive $S_{i-1},S_i$ are cobordant with vertical boundary. Clearly, this can only be the case if for every boundary component $c\in\pi_0(\partial M)$, the intersections $A\cap c$ and $B\cap c$ are diffeomorphic. For a convenient sufficient condition, we say hypersurfaces $A$, $B$ in $M$ \emph{only allow product cobordisms in the boundary} if
\begin{itemize}
    \item their boundaries are disjoint, and
    \item  every cobordism in $\partial M$ between unions of components of $\partial A\cup \partial B$ is a union of product cobordisms and closed components of $\partial M$.
\end{itemize}
We now prove our main theorem under the additional assumption that the hypersurfaces are disjoint.
\begin{prop}[Cobordism of disjoint hypersurfaces]\label{prop:cob_disj}
    Let $A,B$ be disjoint hypersurfaces in an $n$-dimensional manifold $M$ that represent the same homology class in $\h_{n-1}(M,\partial M)$. Then $A$ and $B$ are cobordism-equivalent. 

    If $A$, $B$ only allow product cobordisms in the boundary, $A$ and $B$ are cobordism-equivalent with vertical boundary.
\end{prop}
\begin{proof}
    Consider the hypersurface $\Sigma\coloneq  A\cup B$. We define the graph $\Gamma$ whose vertices are the connected components of $M\setminus \Sigma$ and whose edges are the components of $\Sigma$. A subset $C$ of components of $\Sigma$ induces an orientation of $\Gamma$ by orienting the edges in $C$ to go from the vertex touching their negative side to the one touching their positive side, and the other edges in the reverse way. Conversely, an orientation $O$ of $\Gamma$ corresponds in this way to a hypersurface $\Sigma^O$.

    First consider $\Gamma$ with the orientation $O$ corresponding to $\Sigma^O= A$. If $\gamma$ is a cycle in $\Gamma$, the number $p$ of edges $\gamma$ traverses in the direction of $O$ must equal the number $n$ of edges $\gamma$ traverses against the direction of $O$. Indeed, we can realize $\gamma$ as a loop in $M$ that is transverse to $\Sigma$. Then
    \[
        p-n=[\gamma]\cdot[A\cup B\rev]=[\gamma]\cdot([A]-[B])=0
    \]
    Thus \cref{lem:sinkpushup} yields a sequence of orientations 
    \[
    O=O_0,O_1,\dots,O_k=O\rev
    \]
    of $\Gamma$ such that $O_i$ arises from $O_{i-1}$ by pushing up a sink at a vertex $v_i$. The first claim follows as the closure of $v_i\in\pi_0(M\setminus\Sigma)$ is a cobordism from $\Sigma^{O_{i-1}}$ to $\Sigma^{O_i}$ by \cref{lem:comp_to_cob}, and $O\rev$ is the orientation corresponding to $B$.
    
    Now suppose $A,B$ only allow product cobordisms in the boundary. The vertical boundary~$\partial_v v_i$ is a cobordism between unions of components of $\partial A\cup\partial B$. It follows that $v_i$ is a cobordism with vertical boundary.
\end{proof}
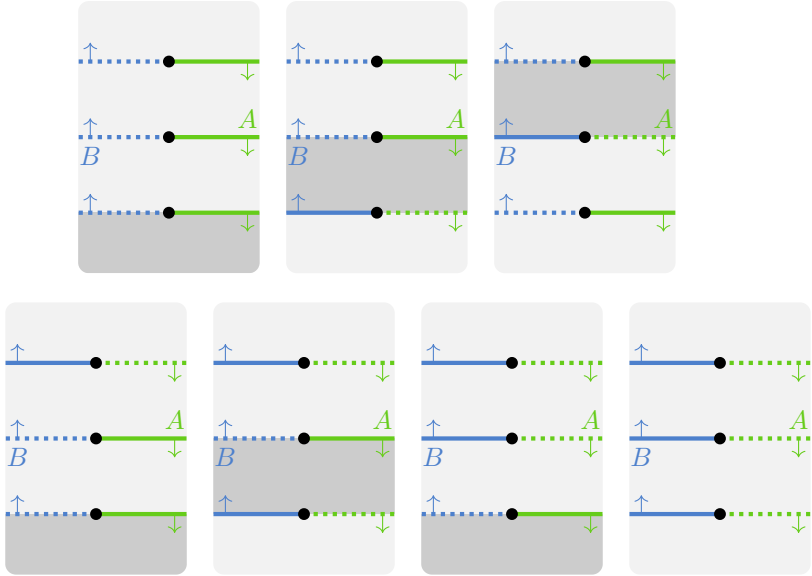
\begin{figure}[H]
    \centering
    \quad\qquad\begin{tikzpicture}[xscale=0.8]
    \clip[rounded corners] (-1.5,-1.8) rectangle (1.5,1.8);
    \fill[black!05] (-1.5,-1.8) rectangle (1.5,1.8);
    \fill[black!20] (-1.5,-1) rectangle (1.5,-1.8);  
    
    \draw[ultra thick, color1] (0,1) -- +(1.5,0); 
    \draw[ultra thick, color2, dotted] (0,1) -- +(-1.5,0); 
    \draw (0,1) pic[ultra thick]{dot};
    \draw[color1,->] (1.3,1) -- +(0,-0.25);
    \draw[color2,->] (-1.3,1) -- +(0,0.25);

    \draw[ultra thick, color1] (0,0) -- +(1.5,0); 
    \draw[ultra thick, color2, dotted] (0,0) -- +(-1.5,0); 
    \draw (0,0) pic[ultra thick]{dot};
    \draw[color1,->] (1.3,0) -- +(0,-0.25);
    \draw[color2,->] (-1.3,0) -- +(0,0.25);

    \draw[ultra thick, color1] (0,-1) -- +(1.5,0); 
    \draw[ultra thick, color2, dotted] (0,-1) -- +(-1.5,0); 
    \draw (0,-1) pic[ultra thick]{dot};
    \draw[color1,->] (1.3,-1) -- +(0,-0.25);
    \draw[color2,->] (-1.3,-1) -- +(0,0.25);
        
    \path (1.3,0) node[above,color1]{$A$} (-1.3,0) node[below,color2]{$B$};
    \end{tikzpicture}\quad
    \begin{tikzpicture}[xscale=0.8]
    \clip[rounded corners] (-1.5,-1.8) rectangle (1.5,1.8);
    \fill[black!05] (-1.5,-1.8) rectangle (1.5,1.8);
    \fill[black!20] (-1.5,0) rectangle (1.5,-1);  
    
    \draw[ultra thick, color1] (0,1) -- +(1.5,0); 
    \draw[ultra thick, color2, dotted] (0,1) -- +(-1.5,0); 
    \draw (0,1) pic[ultra thick]{dot};
    \draw[color1,->] (1.3,1) -- +(0,-0.25);
    \draw[color2,->] (-1.3,1) -- +(0,0.25);

    \draw[ultra thick, color1] (0,0) -- +(1.5,0); 
    \draw[ultra thick, color2, dotted] (0,0) -- +(-1.5,0); 
    \draw (0,0) pic[ultra thick]{dot};
    \draw[color1,->] (1.3,0) -- +(0,-0.25);
    \draw[color2,->] (-1.3,0) -- +(0,0.25);

    \draw[ultra thick, color1, dotted] (0,-1) -- +(1.5,0); 
    \draw[ultra thick, color2] (0,-1) -- +(-1.5,0); 
    \draw (0,-1) pic[ultra thick]{dot};
    \draw[color1,->] (1.3,-1) -- +(0,-0.25);
    \draw[color2,->] (-1.3,-1) -- +(0,0.25);
        
    \path (1.3,0) node[above,color1]{$A$} (-1.3,0) node[below,color2]{$B$};
    \end{tikzpicture}\quad\vspace*{1em}
    \begin{tikzpicture}[xscale=0.8]
    \clip[rounded corners] (-1.5,-1.8) rectangle (1.5,1.8);
    \fill[black!05] (-1.5,-1.8) rectangle (1.5,1.8);
    \fill[black!20] (-1.5,1) rectangle (1.5,0);  
    
    \draw[ultra thick, color1] (0,1) -- +(1.5,0); 
    \draw[ultra thick, color2, dotted] (0,1) -- +(-1.5,0); 
    \draw (0,1) pic[ultra thick]{dot};
    \draw[color1,->] (1.3,1) -- +(0,-0.25);
    \draw[color2,->] (-1.3,1) -- +(0,0.25);

    \draw[ultra thick, color1, dotted] (0,0) -- +(1.5,0); 
    \draw[ultra thick, color2] (0,0) -- +(-1.5,0); 
    \draw (0,0) pic[ultra thick]{dot};
    \draw[color1,->] (1.3,0) -- +(0,-0.25);
    \draw[color2,->] (-1.3,0) -- +(0,0.25);

    \draw[ultra thick, color1] (0,-1) -- +(1.5,0); 
    \draw[ultra thick, color2, dotted] (0,-1) -- +(-1.5,0); 
    \draw (0,-1) pic[ultra thick]{dot};
    \draw[color1,->] (1.3,-1) -- +(0,-0.25);
    \draw[color2,->] (-1.3,-1) -- +(0,0.25);
        
    \path (1.3,0) node[above,color1]{$A$} (-1.3,0) node[below,color2]{$B$};
    \end{tikzpicture}\newline
    \begin{tikzpicture}[xscale=0.8]
    \clip[rounded corners] (-1.5,-1.8) rectangle (1.5,1.8);
    \fill[black!05] (-1.5,-1.8) rectangle (1.5,1.8);
    \fill[black!20] (-1.5,-1) rectangle (1.5,-1.8);  
    
    \draw[ultra thick, color1, dotted] (0,1) -- +(1.5,0); 
    \draw[ultra thick, color2] (0,1) -- +(-1.5,0); 
    \draw (0,1) pic[ultra thick]{dot};
    \draw[color1,->] (1.3,1) -- +(0,-0.25);
    \draw[color2,->] (-1.3,1) -- +(0,0.25);

    \draw[ultra thick, color1] (0,0) -- +(1.5,0); 
    \draw[ultra thick, color2, dotted] (0,0) -- +(-1.5,0); 
    \draw (0,0) pic[ultra thick]{dot};
    \draw[color1,->] (1.3,0) -- +(0,-0.25);
    \draw[color2,->] (-1.3,0) -- +(0,0.25);

    \draw[ultra thick, color1] (0,-1) -- +(1.5,0); 
    \draw[ultra thick, color2, dotted] (0,-1) -- +(-1.5,0); 
    \draw (0,-1) pic[ultra thick]{dot};
    \draw[color1,->] (1.3,-1) -- +(0,-0.25);
    \draw[color2,->] (-1.3,-1) -- +(0,0.25);
        
    \path (1.3,0) node[above,color1]{$A$} (-1.3,0) node[below,color2]{$B$};
    \end{tikzpicture}\quad
    \begin{tikzpicture}[xscale=0.8]
    \clip[rounded corners] (-1.5,-1.8) rectangle (1.5,1.8);
    \fill[black!05] (-1.5,-1.8) rectangle (1.5,1.8);
    \fill[black!20] (-1.5,0) rectangle (1.5,-1);  
    
    \draw[ultra thick, color1, dotted] (0,1) -- +(1.5,0); 
    \draw[ultra thick, color2] (0,1) -- +(-1.5,0); 
    \draw (0,1) pic[ultra thick]{dot};
    \draw[color1,->] (1.3,1) -- +(0,-0.25);
    \draw[color2,->] (-1.3,1) -- +(0,0.25);

    \draw[ultra thick, color1] (0,0) -- +(1.5,0); 
    \draw[ultra thick, color2, dotted] (0,0) -- +(-1.5,0); 
    \draw (0,0) pic[ultra thick]{dot};
    \draw[color1,->] (1.3,0) -- +(0,-0.25);
    \draw[color2,->] (-1.3,0) -- +(0,0.25);

    \draw[ultra thick, color1, dotted] (0,-1) -- +(1.5,0); 
    \draw[ultra thick, color2] (0,-1) -- +(-1.5,0); 
    \draw (0,-1) pic[ultra thick]{dot};
    \draw[color1,->] (1.3,-1) -- +(0,-0.25);
    \draw[color2,->] (-1.3,-1) -- +(0,0.25);
        
    \path (1.3,0) node[above,color1]{$A$} (-1.3,0) node[below,color2]{$B$};
    \end{tikzpicture}\quad
    \begin{tikzpicture}[xscale=0.8]
    \clip[rounded corners] (-1.5,-1.8) rectangle (1.5,1.8);
    \fill[black!05] (-1.5,-1.8) rectangle (1.5,1.8);
    \fill[black!20] (-1.5,-1) rectangle (1.5,-1.8);  
    
    \draw[ultra thick, color1, dotted] (0,1) -- +(1.5,0); 
    \draw[ultra thick, color2] (0,1) -- +(-1.5,0); 
    \draw (0,1) pic[ultra thick]{dot};
    \draw[color1,->] (1.3,1) -- +(0,-0.25);
    \draw[color2,->] (-1.3,1) -- +(0,0.25);

    \draw[ultra thick, color1, dotted] (0,0) -- +(1.5,0); 
    \draw[ultra thick, color2] (0,0) -- +(-1.5,0); 
    \draw (0,0) pic[ultra thick]{dot};
    \draw[color1,->] (1.3,0) -- +(0,-0.25);
    \draw[color2,->] (-1.3,0) -- +(0,0.25);

    \draw[ultra thick, color1] (0,-1) -- +(1.5,0); 
    \draw[ultra thick, color2, dotted] (0,-1) -- +(-1.5,0); 
    \draw (0,-1) pic[ultra thick]{dot};
    \draw[color1,->] (1.3,-1) -- +(0,-0.25);
    \draw[color2,->] (-1.3,-1) -- +(0,0.25);
        
    \path (1.3,0) node[above,color1]{$A$} (-1.3,0) node[below,color2]{$B$};
    \end{tikzpicture}\quad
    \begin{tikzpicture}[xscale=0.8]
    \clip[rounded corners] (-1.5,-1.8) rectangle (1.5,1.8);
    \fill[black!05] (-1.5,-1.8) rectangle (1.5,1.8);
    
    \draw[ultra thick, color1, dotted] (0,1) -- +(1.5,0); 
    \draw[ultra thick, color2] (0,1) -- +(-1.5,0); 
    \draw (0,1) pic[ultra thick]{dot};
    \draw[color1,->] (1.3,1) -- +(0,-0.25);
    \draw[color2,->] (-1.3,1) -- +(0,0.25);

    \draw[ultra thick, color1, dotted] (0,0) -- +(1.5,0); 
    \draw[ultra thick, color2] (0,0) -- +(-1.5,0); 
    \draw (0,0) pic[ultra thick]{dot};
    \draw[color1,->] (1.3,0) -- +(0,-0.25);
    \draw[color2,->] (-1.3,0) -- +(0,0.25);

    \draw[ultra thick, color1, dotted] (0,-1) -- +(1.5,0); 
    \draw[ultra thick, color2] (0,-1) -- +(-1.5,0); 
    \draw (0,-1) pic[ultra thick]{dot};
    \draw[color1,->] (1.3,-1) -- +(0,-0.25);
    \draw[color2,->] (-1.3,-1) -- +(0,0.25);
        
    \path (1.3,0) node[above,color1]{$A$} (-1.3,0) node[below,color2]{$B$};
    \end{tikzpicture}
    
    \caption{The graph corresponding to the example from \cref{fig:exampleknot} is the one we considered in \cref{fig:examplegraph}. This is the resulting sequence of cobordisms from $A$ to $B$.}
\end{figure}

This proposition only proves cobordism-equivalence for hypersurfaces which are already disjoint. The final ingredient we need for the proof of our main result is a theorem of Herrmann and the fourth author \cite[Theorem 1.1]{HQ25} allowing us to make hypersurfaces disjoint:

\begin{thm}[Sequentially disjoint hypersurfaces]\label{thm:HQ}
    Let $A,B$ be hypersurfaces in $M$ representing the same homology class $\phi\in \h_{n-1}(M, \partial M)$. Then there exists a sequence of hypersurfaces in $M$
    \[A = S_0, S_1, \ldots, S_m=B,\]
    all representing $\phi$, such that each two consecutive $S_{i-1}, S_i$ are disjoint. 
    

    If $A$, $B$ only allow product cobordisms in the boundary, we may assume that $S_{i-1}$, $S_i$ also only allow product cobordisms in the boundary.
\end{thm}
\begin{proof}
    We first reduce the problem to the case where the intersection of $A$ and~$B$ is transverse. 
    To that end, let $A_+$ be a parallel copy of $A$ (in particular, $A_+ \cap A = \emptyset$), and then by a general position argument perturb $A_+$ by a small isotopy to produce a hypersurface $A'$ that is still disjoint from $A$ and intersects~$B$ transversely. $A'$ is isotopic, and hence homologous, to $A$, so taking $S_1=A'$, we are left to find sequentially disjoint homologous hypersurfaces from $A'$ to $B$.

    Working under the assumption that $A$ and $B$ are transverse, we prove the result by induction on the number of components of $A\cap B$. Of course if this number is $0$ there is nothing to prove. If $|\pi_0(A \cap B)| \ge 1$, we will show that there is a hypersurface $S$ representing $\phi$ that intersects both $A$ and $B$ transversely, and for which
    \[\tag{$\star$} |\pi_0(S\cap A)| \le \tfrac 12|\pi_0(A \cap B)| \quad \text{and} \quad |\pi_0(S\cap B)| \le \tfrac 12 |\pi_0(A\cap B)|. \]
    Once such $S$ is constructed, the induction hypothesis yields sequences of hypersurfaces from $A$ to $S$ and from $S$ to $B$, which can be concatenated to produce a sequence from $A$ to $B$.

    The first step to produce $S$ is to construct a so-called ``oriented surgery'' of $A$ and $B$. This is a hypersurface $\Sigma_0$ obtained from $A\cup B$ by replacing a small neighborhood of $A\cap B$ with a ``ramp'' that pairs each sheet of $A$ with a sheet of~$B$ in a manner consistent with orientations; \cref{fig:oriented_surgery} (center) gives a visual idea of the construction. $\Sigma_0$ inherits a canonical orientation from $A$  and $B$. We then push $\Sigma_0$ in the positive direction of its normal bundle, as in \cref{fig:oriented_surgery} (right), to obtain an isotopic hypersurface $\Sigma$ transverse to $A$ and $B$. 

    \begin{figure}[H]
        \centering
        \begin{tikzpicture}[scale=0.9]
        \fill[black!05,rounded corners] (-2,-2) rectangle (2,2);
        \draw[ultra thick, color1] (-2,0) -- (2,0);
        \draw[ultra thick, color2] (0,-2) -- (0,2);
        \draw[color1,->] (1.7,0) node[below]{$A$} -- +(0,0.277);
        \draw[color2,->] (0,1.7) node[right]{$B$} -- +(-0.277,0);
        \path (0,0) pic[ultra thick]{dot};
        \end{tikzpicture}\quad
        \begin{tikzpicture}[scale=0.9]
        \fill[black!05,rounded corners] (-2,-2) rectangle (2,2);
        \draw[thick] (-2,0) -- (-1,0) .. controls (0,0) .. (0,1) -- (0,2);
        \draw[thick,rotate=180] (-2,0) -- (-1,0) .. controls (0,0) .. (0,1) -- (0,2);
        \draw[->] (1.7,0) -- +(0,0.277);
        \draw[->] (0,1.7) -- +(-0.277,0);
        \path (-1.7,0) node[above]{$\Sigma_0$};
        \end{tikzpicture}\quad
        \begin{tikzpicture}[scale=0.9]
        \clip (-2,-2) rectangle (2,2);
        \fill[black!05,rounded corners] (-2,-2) rectangle (2,2);
        \draw[ultra thick, color1] (-2,0) -- (2,0);
        \draw[ultra thick, color2] (0,-2) -- (0,2);
        \path (0,0) pic[ultra thick]{dot};
        \draw[shift={(-0.4,0.4)},thick] (-2,0) -- (-1,0) .. controls (0,0) .. (0,1) -- (0,2);
        \draw[shift={(-0.4,0.4)},rotate=180,thick] (-2.5,0) -- (-1,0) .. controls (0,0) .. (0,1) -- (0,2.5);
        \draw[->] (1.7,0.4) -- +(0,0.277);
        \draw[->] (-0.4,1.7) -- +(-0.277,0);
        \path (-1.7,0.4) node[above]{$\Sigma$};
        \end{tikzpicture}
    \caption{The construction of $\Sigma$ from $A$ and $B$.\\Left: a neighborhood of the intersection $A\cap B$.\\Center: the oriented surgery $\Sigma_0$ is obtained from $A\cup B$ by resolving the intersection in a manner consistent with all orientations.\\Right: $\Sigma$ is a push-off of $\Sigma_0$ in the positive normal direction.}
        \label{fig:oriented_surgery}
    \end{figure}
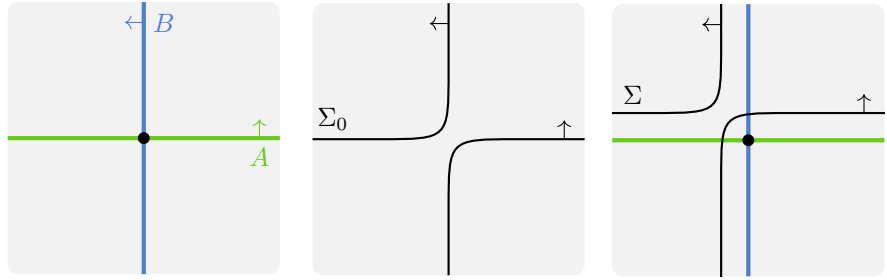

    The first observation we make about $\Sigma$ is that it represents the homology class~$2\phi$. Our next goal is to decompose $\Sigma$ as a disjoint union of two hypersurfaces~$S,T$, each representing $\phi$, and such that one of them has controlled intersection with $A$ and $B$ as in $(\star)$.
    Consider the directed graph $\Gamma$ whose vertices are the connected components of~$M\setminus \Sigma$, and whose edges are the components~$C$ of~$\Sigma$ oriented from the component of~$M\setminus\Sigma$ touching the negative side of $C$ to the one touching its positive side.

    The graph $\Gamma$ is bipartite by \cref{lem:crit_bip} since it has no cycles of odd length.
    Indeed, each cycle in $\Gamma$ may be used to produce a loop $\gamma$ in~$M$ that visits the components of $M\setminus \Sigma$ in the same sequence and intersects~$\Sigma$ transversely. In particular, a cycle of odd length would produce a loop~$\gamma$ such that the intersection product $[\gamma] \cdot [\Sigma]$ is odd, a contradiction to the fact that $[\Sigma] = 2\phi$. 

    So now let $W$ be the union of all components of $M\setminus \Sigma$ that make up one of the two subsets giving a bipartition of $\operatorname{V}(\Gamma)$, and define~$T$ (respectively~$S$) as the union of components of $\Sigma$ whose positive (respectively negative) side touches $W$. In other words, $W$ is a cobordism from $S$ to $T$, and so $[S]=[T]$ by \cref{lem:easymain}. Therefore $2[S] = [S]+[T]  =[\Sigma] = 2\phi$, implying $[S]=\phi$ as $\h_{n-1}(M, \partial M) \cong \h^1(M)\cong \operatorname{Hom}(\h_1(M), \ZZ)$ is torsion-free.

    All that is left to do is show that one of $S,T$ satisfies $(\star)$, which may be seen by taking a closer look at \cref{fig:oriented_surgery} (right). The key observation is that each component of $A\cap B$ produces precisely one component in $\Sigma \cap A$ and one in $\Sigma\cap B$, and both of these lie in the \enquote{lower} sheet of~$\Sigma$ -- in particular, they both lie in $S$ or both lie in $T$. Thus, we have
    \[ |\pi_0(S\cap A)| = |\pi_0(S \cap B)| \quad\text{and} \quad |\pi_0(T\cap A)| = |\pi_0(T \cap B)|,\]
    with these two numbers adding up to $|\pi_0(A \cap B)|$. Thus, at least one is bounded above by $\tfrac 12|\pi_0(A \cap B)|$. Assuming without loss of generality it is the one coming from $S$, we conclude that~$(\star)$ holds, finishing the proof of the first part.
    
    Now assume $A$, $B$ only allow product cobordisms in the boundary. By construction $\partial S$ consists of slightly shifted components of $\partial A\cup \partial B$ where we have ensured all shifts are small enough that they do not intersect $\partial A\cup \partial B$. Hence, cobordisms between components of $\partial A\cup\partial S$ are cobordisms between components of $\partial A\cup \partial B$ with a small product cobordism added or removed. Note that removing a product cobordism from a product cobordism leaves a product cobordism by the uniqueness of boundary collars. It follows that $A,S$ only allow product cobordisms in the boundary. Analogously, the same holds for $S,B$. 
    \end{proof}
    With this we can prove our main theorem with a strengthening we will need for the application to Seifert surfaces.
    \begin{thm}[Homology and cobordism-equivalence]\label{thm:main}
        Two hypersurfaces $A$,\nobreak\ $B$ in an $n$-dimensional manifold $M$ are cobordism-equivalent if and only if they represent the same homology class in $\h_{n-1}(M,\partial M)$.
        
            If $A$, $B$ only allow product cobordisms in the boundary, the above conditions are also equivalent to $A$ and $B$ being cobordism-equivalent with vertical boundary. 
    \end{thm}
    \begin{proof}
        One implication is \cref{lem:easymain}, the other follows from \cref{thm:HQ,prop:cob_disj}.
        %
    \end{proof}

\section{Cobordisms and surgery}
We now apply the theory of handle decompositions to the cobordisms from the main theorem. A handle decomposition of a cobordism from~$A$ to~$B$ corresponds to a sequence of surgeries from~$A$ to~$B$. In light of this, a hypersurface~$B$ arises from a hypersurface~$A$ in~$M$ by \emph{(embedded) surgery of index~$k$} if there exists a smooth embedding $\tau\colon \DD^k\times \DD^{n-k}\to M\setminus\partial M$ such that
\begin{itemize}
    \item $\tau(\DD^k\times \DD^{n-k})\cap A=\tau(\partial\DD^k\times \DD^{n-k})\subseteq A\setminus\partial A$, and
    \item $\tau(\DD^k\times \DD^{n-k})\cap B=\tau(\DD^k\times \partial\DD^{n-k})\subseteq B\setminus\partial B$.
\end{itemize}
The vertical part of the cobordisms constructed in the second part of \cref{prop:cob_disj} need not be a product cobordism because it could have closed components. We therefore cannot directly construct handle decompositions, but need to allow an additional operation: A hypersurface $B$ arises from a hypersurface $A$ in $M$ by \emph{removing a boundary-parallel component} if there exist a boundary component $C$ of $M$ and a smooth embedding $f\colon C\times[0,1]\to M\setminus B$ that restricts to the identity on $C\times\{0\}$ and an orientation-preserving diffeomorphism onto a component of $A$ on $C\times\{1\}$. Here we equip $C$ with the boundary orientation. 

We obtain the following corollary. 
\begin{cor}[Equivalence through surgery]\label{cor:equiv_surg}
    Let $A$, $B$ be hypersurfaces in an $n$-dimensional manifold $M$ that only allow product cobordisms in the boundary and represent the same homology class in $\h_{n-1}(M,\partial M)$. There exists a sequence of smooth isotopies, surgeries and removing boundary-parallel components from $A$ to $B$. 
\end{cor}

\begin{proof}
    By \cref{thm:main} $A$ and $B$ are cobordism-equivalent through cobordisms with vertical boundary. Hence, it suffices to consider the case when there exists a cobordism $W$ with vertical boundary from $A$ to $B$. Then $\partial_vW$ is given by a product cobordism and a union $C$ of boundary components of $M$. Slightly push~$C$ into $M$. Since the orientation of $W$ is induced by the orientation of $M$, $W$ now gives a cobordism from $A$ to $B\cup C$.
    
    By \cref{thm:handle_decomp} $W$ has a handle decomposition as a cobordism from $A$ to $B\cup C$ . As surgery is the effect of a handle attachment in the boundary, we obtain a sequence of surgeries from a hypersurface isotopic to $A$ to $B\cup C$. By construction, $B$ arises from $B\cup C$ by removing boundary-parallel components.
\end{proof}
Just as with handle decompositions we may reorder these surgeries into a standard form. Note however that this is not a formal consequence of the corresponding statement about handle decompositions since the cobordisms in the last proof may interact with each other.
\begin{lem}[Rearranging surgeries]\label{lem:standard}
    Let $A$, $B$ be hypersurfaces in an $n$-dimensional manifold~$M$. Then every sequence of surgeries and smooth isotopies from $A$ to $B$ may be rearranged into a sequence with the same number of surgeries, and in which the surgery indices occur in increasing order.
\end{lem}
\begin{proof}
    This is essentially proven in the same manner as the analogous statements about handle decompositions (see \cite[Lemma 5.2.1]{Wall}). Note that the surgeries being embedded does not present an obstruction as the hypersurfaces allow bicollars. 
\end{proof}
The property of only allowing product cobordisms in the boundary is somewhat contrived. One setting where it arises naturally are 3-manifolds $M$ with \emph{toroidal boundary}, i.e.\ $\partial M$ is a (possibly empty) union of tori. In this case, a hypersurface $A$ in $M$ is \emph{spanning}, if
\begin{itemize}
    \item every component of $\partial M$ contains a component of $\partial A$,
    \item no component of $\partial A$ is null-homologous in $\partial M$, and
    \item components of $\partial A$ in the same boundary torus of $M$ are homologous in it.
\end{itemize}
This leads to our motivating example when $M$ is the complement of an open tubular neighborhood of a link $L$ in~$S^3$ and $A$ is the part of a Seifert surface for $L$ lying in $M$. To align terminology, we call a surgery of index 1 on a hypersurface a \emph{tube attachment} and the reverse operation a \emph{tube removal}. To obtain the desired result, it remains to ensure that we can model all surgeries in dimension 3 by tube attachments and tube removals. This can be arranged if the hypersurfaces are connected and represent a non-zero homology class. 
\begin{lem}[Equivalence through tube attachments]\label{lem:equiv_tube}
    Let $M$ be a connected 3-di\-men\-sio\-nal manifold with toroidal boundary. Let $A$, $B$ be connected spanning hypersurfaces in $M$ representing the same non-zero homology class in $\h_{2}(M,\partial M)$. Then $B$ arises from $A$ by a sequence of smooth isotopies, tube attachments and tube removals with the tube attachments occurring before the tube removals. In particular, all intermediate surfaces are also connected.
\end{lem}
\begin{proof}
Since the homology of $\partial M$ is the sum of the homology of the boundary components, it follows that components of $\partial A$ or $\partial B$ in the same component of~$\partial M$ are homologous in it. By the classification of smooth curves in a torus we may therefore assume after a smooth isotopy that $\partial A$ and $\partial B$ are disjoint, in which case $A,B$ only allow product cobordisms in the boundary. 

Now apply \cref{cor:equiv_surg} to obtain a sequence from $A$ to $B$ consisting of smooth isotopies and surgeries of index
\begin{itemize}
    \item 0, i.e.\ adding the boundary of a $3$-ball in $M$,
    \item 1, i.e.\ tube attachments,
    \item 2, i.e.\ tube removals by the usual handle duality argument, and
    \item 3, i.e.\ removing the boundary of a $3$-ball in $M$.
\end{itemize}
Note that this sequence cannot contain any removals of boundary-parallel components as the intermediate surfaces all represent the same homology class and thereby have non-trivial intersection with every boundary component of $M$.

By \cref{lem:standard} we may assume that the indices of the surgeries are increasing. It remains to express the surgeries of index 0 and 3 using tube attachments and tube removals. 

Since $A$ is connected, $M\setminus A$ has at most two components. If $M\setminus A$ has precisely two components, one of them is a cobordism from $A$ to $\emptyset$, implying by \cref{lem:easymain} that $[A]\in\h_2(M,\partial M)$ is zero. Contradiction! Hence, $M\setminus A$ is connected. Now let $X\subseteq M$ be a sphere resulting from surgery of index 0 on~$A$. Then there exists a smooth embedding $\gamma\colon[0,1]\to M$ such that
\begin{itemize}
    \item $A\cap\gamma([0,1])=\gamma(\{0\})$ and this intersection is transverse with $\gamma$ pointing in the normal direction of $A$, and
    \item $X\cap\gamma([0,1])=\gamma(\{1\})$ and this intersection is transverse with $\gamma$ pointing against the normal direction of $X$,
\end{itemize}
or with the direction of $\gamma$ reversed on both $A$ and $X$. Then we can obtain $A\cup X$ by first isotoping~$A$ and then removing a tube given by a thickening of $\gamma$.
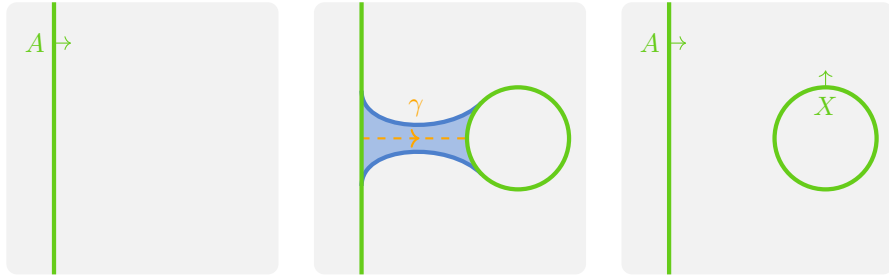
\begin{figure}[H]
    \centering
    \begin{tikzpicture}[scale=0.9]
        \fill[black!05,rounded corners] (-2,-2) rectangle (2,2);
        \draw[color1,ultra thick] (-1.3,2) -- +(0,-4);        
        \draw[color1,->] (-1.3,1.4) node[left]{$A$} -- +(0.25,0);
    \end{tikzpicture}   \quad
    \begin{tikzpicture}[scale=0.9]
        \fill[black!05,rounded corners] (-2,-2) rectangle (2,2);
        \begin{scope}[even odd rule]
            \clip (-2,-2) rectangle (2,2) (1,0) circle(0.75);
            \fill[ultra thick,white] (-1.3,0.7) .. controls +(0,-0.6) and +(-0.5,-0.5) .. (0.47,0.53) -- (0.47,-0.53) .. controls +(-0.5,0.5) and +(0,0.6) .. (-1.3,-0.7) -- cycle;
            \fill[ultra thick,color2!50] (-1.3,0.7) .. controls +(0,-0.6) and +(-0.5,-0.5) .. (0.47,0.53) -- (0.47,-0.53) .. controls +(-0.5,0.5) and +(0,0.6) .. (-1.3,-0.7) -- cycle;
            \draw[ultra thick,color2] (-1.3,0.7) .. controls +(0,-0.6) and +(-0.5,-0.5) .. (0.47,0.53)  (0.47,-0.53) .. controls +(-0.5,0.5) and +(0,0.6) .. (-1.3,-0.7);
        \end{scope}
        \draw[color3,dashed,thick,->-=0.55] (-1.3,0) -- (0.25,0);
        \draw[color3] (-0.5,0) node[above=0.5em]{$\gamma$};
        \draw[color1,ultra thick] (-1.3,2) -- +(0,-4) (1,0) circle(0.75);
    \end{tikzpicture}   \quad
    \begin{tikzpicture}[scale=0.9]
        \fill[black!05,rounded corners] (-2,-2) rectangle (2,2);
        \draw[color1,ultra thick] (-1.3,2) -- +(0,-4) (1,0) circle(0.75);
        \draw[color1,->] (1,0.75) node[below]{$X$} -- +(0,0.25);
        \draw[color1,->] (-1.3,1.4) node[left]{$A$} -- +(0.25,0);
    \end{tikzpicture}   
\end{figure}
Dually, we may express the surgeries of index 3 by tube attachments. We obtain a sequence of smooth isotopies, tube attachments and tube removals from $A$ to $B$. After applying \cref{lem:standard} again, we may assume that the tube attachments occur before the tube removals. 
\end{proof}
The last step is to bridge the gap between Seifert surfaces with boundary of a link in the interior of a manifold and properly embedded hypersurfaces by removing small enough tubular neighborhood of the link. To ensure that the resulting hypersurfaces represent the same homology class, we use that the surrounding manifold is a rational homology sphere.
\begin{cor}[Tube equivalence for Seifert surfaces]\label{cor:seifert_tube}
    Let $A$, $B$ be Seifert surfaces for a link $L$ in a rational homology sphere $X$. Then $B$ arises from $A$ by a sequence of smooth isotopies, tube at\-tach\-ments and tube removals with the tube attachments occurring before the tube removals. In particular, all intermediate surfaces are also Seifert surfaces.
\end{cor}
\begin{proof}
    After a smooth isotopy of $A$ and $B$ we may assume that there exists a tubular neighborhood $\tau\colon L\times\DD^2\to X$ such that 
    \begin{align*}
        A\cap\tau(L\times\DD^2)&=\tau(L\times\{t\cdot a\mid 0\leq t\leq 1\})\\
        B\cap\tau(L\times\DD^2)&=\tau(L\times\{t\cdot b\mid 0\leq t\leq 1\})
    \end{align*}
    for distinct $a,b\in\partial\DD^2$. Then $M\coloneqq X\setminus\tau(L\times\OD^2)$ is a connected 3-dimensional manifold with toroidal boundary and $\hat A\coloneqq M\cap A$, $\hat B\coloneqq M\cap B$ are connected spanning hypersurfaces in $M$.

    To apply \cref{lem:equiv_tube} it remains to show that $[\hat A]=[\hat B]$ in $\h_2(M,\partial M)$. As $\partial\hat A$ and $\partial\hat B$ are smoothly isotopic in $\partial M$, it suffices to show that $\h_2(M,\partial M)\to\h_1(\partial M)$ is injective. Since both are free $\ZZ$-modules, we may consider rational coefficients where we have an exact sequence
    \[
    0\!\to\!\h_3(M,\partial M;\QQ)\!\to\!\h_2(\partial M;\QQ)\!\to\!\h_2(M;\QQ)\!\to\!\h_2(M,\partial M;\QQ)\!\to\!\h_1(\partial M;\QQ).
    \]
    Note that 
    \begin{itemize}
        \item $\dim\h_3(M,\partial M;\QQ)=1$,
        \item $\dim\h_2(\partial M;\QQ)=k$, where $k$ is the number of components of $L$, and
        \item $\dim\h_2(M;\QQ)=k-1$ by Alexander duality in the rational homology sphere~$X$.
    \end{itemize}
    Hence, $\h_2(\partial M;\QQ)\to\h_2(M;\QQ)$ is surjective and $\h_2(M,\partial M;\QQ)\to\h_1(\partial M;\QQ)$ is injective as desired.
\end{proof}


\printbibliography

\flushleft
---------

\textsc{Stefan Friedl},   \texttt{\href{mailto:sfriedl@gmail.com}{sfriedl@gmail.com}}

\textsc{Tobias Hirsch},   \texttt{\href{mailto:Tobias.Hirsch@ur.de}{Tobias.Hirsch@ur.de}}

\textsc{Clayton McDonald},   \texttt{\href{mailto:claytkm@gmail.com}{claytkm@gmail.com}}

\textsc{José Pedro Quintanilha},   \texttt{\href{mailto:jquintanilha@mathi.uni-heidelberg.de}{jquintanilha@mathi.uni-heidelberg.de}}

\textsc{Daniel Zach},   \texttt{\href{mailto:Daniel.Zach@stud.uni-regensburg.de}{Daniel.Zach@stud.uni-regensburg.de}}

\end{document}